\documentclass[12pt,a4paper]{amsart}
\usepackage{amsmath,amssymb,amsthm,mathtools}
\usepackage{amsfonts,amsxtra,amsthm,tikz}
\usepackage[headings]{fullpage}
\usepackage{algorithm}

\RequirePackage[colorlinks=false]{hyperref}
\usepackage[noend]{algpseudocode}

\usetikzlibrary{arrows,positioning}
\usetikzlibrary{decorations.markings}
\tikzstyle arrowstyle=[scale=1.5]
\tikzstyle directed=[postaction={decorate,decoration={markings,
    mark=at position .725 with {\arrow[arrowstyle]{stealth}}}}]
\tikzstyle directedsp=[postaction={decorate,decoration={markings,
    mark=at position .375 with {\arrow[arrowstyle]{stealth}}}}]

\newtheorem{theorem}{Theorem}[section]
\newtheorem{proposition}[theorem]{Proposition} 
\newtheorem{lemma}[theorem]{Lemma}              
\newtheorem{corollary}[theorem]{Corollary}         
\newtheorem{conjecture}[theorem]{Conjecture}         

\theoremstyle{definition}
\newtheorem{definition}[theorem]{Definition}        
\newtheorem{remark}[theorem]{Remark}              

\theoremstyle{remark}
\newtheorem{example}[theorem]{Example}          

\numberwithin{equation}{section}

\newcommand\stt{\rule[-2pt]{0pt}{5pt}}
\newcommand\cx[1]{\multicolumn{1}{|c|}{\!#1\!}}
\newcommand{\bnv}{\textup{\texttt{burnt\_vertices}}}
\newcommand{\damp}{\texttt{dampened\_edges}}
\newcommand{\tree}{\texttt{tree\_edges}}
\newcommand{\pf}{\mathbf{a}}
\newcommand{\pg}{\mathbf{b}}
\newcommand\cA{\mathcal{A}}
\newcommand\cM{\mathcal{M}}
\newcommand\cG{\mathcal{G}}
\newcommand\cS{\mathcal{S}}
\newcommand\cR{\mathcal{R}}
\newcommand\bG{\overline{\mathcal{\strut G}}}
\newcommand\bV{\overline{\strut V}}
\newcommand\bA{\overline{\strut A}}

\newcommand\run{r}

\DeclareMathOperator{\Ish}{Ish}
\DeclareMathOperator{\Shi}{Shi}
\DeclareMathOperator{\Cox}{Cox}
\DeclareMathOperator{\cN}{\textup{\texttt{neighbors}}}

\DeclareMathOperator{\ZP}{\mathcal{ZP}}
\DeclareMathOperator{\ZI}{\mathcal{ZI}}

\newcommand\tin{(1,n)}

\newcommand\ba{\mathbf{a}}
\newcommand\bb{\mathbf{b}}

\newcommand\be{\mathbf{e}}
\newcommand\bz{\boldsymbol{0}}

\newcommand\bw{\mathbf{w}}
\newcommand\bu{\mathbf{u}}
\newcommand\bx{\mathbf{x}}
\newcommand\Reals{\mathbb{R}}
\newcommand\N{\mathbb{N}}

\newcommand{\PF}{\mathsf{PF}}
\newcommand{\IPF}{\mathsf{IPF}}
\DeclareMathOperator{\dfs}{DFS}

\newcommand\tz{\tilde{z}}
\newcommand\tZ{\widetilde{Z}}

\newcommand\tba{\mathbf{\tilde{a}}}
\newcommand\bap{\mathbf{a'}}
\newcommand\tbw{\mathbf{\widetilde{w}}}
\newcommand\tw{\widetilde{w}}
\newcommand\bup{\mathbf{u'}} 

\newcommand\tI{\widetilde{\mathcal{I}}}
\newcommand\Ip{\mathcal{I}'}
\newcommand\fI{\mathcal{I}}

\newcommand\lbdp{\lambda'}

\title{Between Shi and Ish}

\author[R.~Duarte]{Rui Duarte} \address{CIDMA and Department of Mathematics, University of Aveiro, 3810-193 Aveiro, Portugal}
\email{rduarte@ua.pt}

\author[A.~Guedes~de~Oliveira]{Ant\'onio Guedes de Oliveira}
\address{CMUP and Department of Mathematics, Faculty of Sciences, University of Porto, 4169-007 Porto, Portugal}
\email{agoliv@fc.up.pt}

\begin{document}

\begin{abstract}
We introduce a new family of hyperplane arrangements in dimension $n\geq3$ that includes both the Shi arrangement and the Ish arrangement. We prove that all the members of a given subfamily have the same number of regions --- the connected components of the complement of the union of the hyperplanes --- which can be \emph{bijectively} labeled with the Pak-Stanley labeling. In addition, we show that, in the cases of the Shi and the Ish arrangements, the number of labels with \emph{reverse centers} of a given length is equal, and conjecture that the same happens with all of the members of the family.
\end{abstract}
\maketitle

\section{Introduction}
\noindent
In this paper we introduce a family of hyperplane arrangements in general dimension ``between Ish and Shi'', that is,
formed by hyperplanes that are hyperplanes of the Shi arrangement or hyperplanes of the Ish arrangement,
all of the same dimension.

More precisely,  we consider, for an integer $n\geq3$, hyperplanes of $\Reals^n$ of three different types.
Let, for $1\leq i<j\leq n$,
\begin{align*}
C_{ij}&=\big\{(x_1,\dotsc,x_n)\in\Reals^n\mid x_i=x_j\big\}\,;\\
S_{ij}&=\big\{(x_1,\dotsc,x_n)\in\Reals^n\mid x_i=x_j+1\big\}\,;\\
I_{ij}&=\big\{(x_1,\dotsc,x_n)\in\Reals^n\mid x_1=x_j+i\big\}\,.\\
\end{align*} \\[-30pt]
Note that the $n$-dimensional Coxeter arrangement is
$\Cox_n=\bigcup_{1\leq i<j\leq n} C_{ij}$, the $n$-dimen\-sion\-al Shi arrangement is
$\Shi_n=\Cox_n\cup \bigcup_{1\leq i<j\leq n} S_{ij}$ and the $n$-dimensional Ish arrangement,
recently introduced by Armstrong (Cf. \cite{AR}), is
$\strut\Ish_n=\Cox_n\cup \bigcup_{1\leq i<j\leq n} I_{ij}$.

Set $[n]:=\{1,\dotsc,n\}$ 
and $(k,n)=\{k+1,\dotsc,n-1\}$ for $1\leq k<n$,
and define, for any $X\subseteq\tin$,
$$\cA^X:=\Cox_n\cup \bigcup_{\substack{i\in X\\i<j\leq n}} S_{ij}\cup \bigcup_{\substack{i\in[n]\setminus X\\i<j\leq n}} I_{ij}\,,$$
so that $\Shi_n=\cA^{\tin}$ and $\Ish_n=\cA^\varnothing=\cA^{(n-1,n)}$.

We study the arrangements of form $\cA^X$ for $X\subseteq(1,n)$, with special interest in the Pak-Stanley labelings of the regions of the arrangements.
The labels are  \emph{$\cG$-parking functions} for special directed multi-graphs $\cG$ as defined by Mazin \cite{Mazin}.
In particular, we show that, in the case where $X=(k,n)$ for some $1\leq k<n$, there are $(n+1)^{n-1}$ regions which are  \emph{bijectively} labeled.

The notion of $G$-parking function was introduced by Postnikov and Shapiro in the construction of two algebras related to a general \emph{undirected} graph $G$
\cite{PosSha}. Later, Hopkins and Perkinson \cite{HP} showed that the labels of the Pak-Stanley labeling of the regions of a given hyperplane arrangement defined by $G$ are exactly the $G$-parking functions, a fact that had been conjectured by Duval, Klivans and Martin \cite{DKM}.
Recently, Mazin \cite{Mazin} generalized this result to a very general class of hyperplane arrangements, with a similar concept based on a general \emph{directed}
\emph{multi}-graph $\cG$. Whereas Hopkins and Perkinson's hyperplane arrangements include for example the (original) multidimensional Shi arrangement, Mazin's hyperplane arrangements include  the multidimensional $k$-Shi arrangement, the multidimensional Ish arrangement and in fact all the arrangements we consider here.

We start this study  in Section~2 by showing (Cf. \textbf{Theorem~\ref{teor1.1}}) that all arrangements of form $\cA^{(k,n)}$ ($1\leq k<n$)
 have the same \emph{characteristic polynomial}, namely
$$\chi(q)=q(q-n)^{n-1}\,,$$
from which it follows that they all have the same number of regions as well as  the same number of relatively bounded regions, namely $(n+1)^{n-1}$ and $(n-1)^{n-1}$,  by a famous result of Zaslavsky \cite{Zas}.

Since, by a result of Mazin \cite[Theorem 3.1.]{Mazin} the corresponding labels are exactly the $\cG$-parking functions, by showing that they are also $(n+1)^{n-1}$ we prove  (Cf. \textbf{Theorem~\ref{bij2}}) that the labelings are bijective.

For this purpose, in Section 3  we extend to directed multi-graphs a result of Postnikov and Shapiro \cite[Theorem 2.1.]{PosSha} which says (Cf. \textbf{Proposition~\ref{bij}}) that the number of $G$-parking functions is the number of  spanning trees of $G$. More precisely, we extend Perkinson, Yang and Yu's depth-first search version of Dhar's burning algorithm \cite{PYY}, which provides an explicit bijection between both sets, to directed multi-graphs. We also obtain from this algorithm  (Cf. \textbf{Proposition~\ref{prop.ish}}) a characterization of the labels of the $\Ish_n$ arrangement (the \emph{Ish-parking functions} of dimension $n$).

In Section 4 we show that the number of parking functions with \emph{reverse center} (which is essentially the \emph{center} \cite{DGO} of the reverse word --- see Definition~\ref{def.crc}) of a given length  is exactly  the number of Ish-parking functions with reverse center of the same length. This result follows in the direction
of a previous paper of ours \cite{DGO2}, on which it is based, and paves the way to \textbf{Conjecture~\ref{conj}}, which says that the same happens when the Ish-parking functions are replaced by the $\cG$-parking functions corresponding to the new arrangements introduced here.
However, opposite to what happens with the Shi and the Ish arrangements (and, in general, with the arrangements of form $\cA^{(k,n)}$),  the number of $\cG$-parking functions can be less than the number of regions.
In this paper, the Pak-Stanley labeling used in the Shi arrangement is a variation of the original definition by Pak and Stanley. In an Appendix, we explain how to use for this variation the construction developed in another paper of ours \cite{DGO} for inverting  the original Pak-Stanley labeling.
\medskip

\section{The characteristic polynomial}
In the following, we evaluate the \emph{characteristic polynomial}  $\chi(\cA^X,q)$ in the case where $X=(k,n)$. We show that, in this particular case, $\chi(\cA^{(k,n)},q)=\chi(\Shi_n,q)=\chi(\Ish_n,q)$. In other words, the characteristic polynomial of $\cA^{(k,n)}$ does not depend on 
$1\leq k<n$.

For this purpose, we use the finite field method \cite{Athan,Ardila}, which means that
we evaluate $\chi(\cA^{(k,n)},q)$ for any sufficiently large prime number $q$ by counting
the number of elements of the set
$X_q$ formed by the elements $(x_1,\dotsc,x_n)\in\mathbb{F}_q^n$ that verify, for each 
$1\leq i<j\leq n$, the conditions
\begin{equation*}
x_i\neq x_j \wedge
\begin{cases}
x_1\neq x_j+i,& \text{ if $i\leq k$;}\\
x_i\neq x_j+1,& \text{ if $i>k$.}\end{cases}
\end{equation*}
For convenience sake, fixed $n$, we see the elements of
$$S^n_q=\big\{(x_1,\dotsc,x_n)\in\mathbb{F}_q^n\mid\forall 1\leq i<j\leq n\,,\ x_i\neq x_j \big\}$$
as injective labelings in $\mathbb{F}_q$ of the elements of $[n]$ \cite{Ardila,Stan2}.
For instance, for $n=5$ and $q=13$, the labeling represented below corresponds to
$\bx:=(2,9,3,11,12)\in S^5_{13}$. Note that if $(x_1,\dotsc,x_5)=\bx$ then $x_1=x_5+3$.
Hence,  $\bx\in I_{3,5}$ and so $\bx\notin X_{13}$ for $X=\varnothing$.
On the other hand, for no $1\leq i<j\leq 5$ is $x_i=x_j+1$, and so $\bx\in X_{13}$ for $X=(1,5)$.

\smallskip
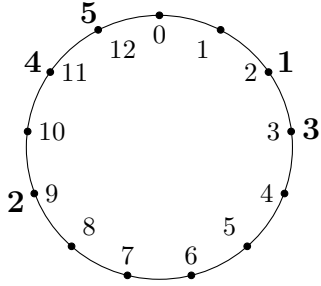
\begin{center}
\begin{figure}[ht]
\begin{tikzpicture}[scale=1.75]
\draw  circle [radius=1];
\coordinate [label=below:{\footnotesize$0$}] (a0) at (0., 1.);
\coordinate [label=below right:{\footnotesize$12$}] (a12) at (-0.46, 0.89);
\coordinate [label=right:{\footnotesize$11$}] (a11) at (-0.82, 0.57);
\coordinate [label=right:{\footnotesize$10$}] (a10) at (-0.99, 0.12);
\coordinate [label=right:{\footnotesize$9$}] (a9) at (-0.94, -0.35);
\coordinate [label=above right:{\footnotesize$8$}] (a8) at (-0.66, -0.75);
\coordinate [label=above :{\footnotesize$7$}] (a7) at (-0.24, -0.97);
\coordinate [label=above :{\footnotesize$6$}] (a6) at (0.24, -0.97);
\coordinate [label=above left:{\footnotesize$5$}] (a5) at (0.66, -0.75);
\coordinate [label=left:{\footnotesize$4$}] (a4) at (0.94, -0.35);
\coordinate [label=left:{\footnotesize$3$}] (a3) at (0.99, 0.12);
\coordinate [label=left:{\footnotesize$2$}] (a2) at (0.82, 0.57);
\coordinate [label=below left:{\footnotesize$1$}] (a1) at (0.46, 0.89);
\draw [fill] (a0) circle [radius=.025];
\draw [fill] (a1) circle [radius=.025];
\draw [fill] (a2) circle [radius=.025];
\draw [fill] (a3) circle [radius=.025];
\draw [fill] (a4) circle [radius=.025];
\draw [fill] (a5) circle [radius=.025];
\draw [fill] (a6) circle [radius=.025];
\draw [fill] (a7) circle [radius=.025];
\draw [fill] (a8) circle [radius=.025];
\draw [fill] (a9) circle [radius=.025];
\draw [fill] (a10) circle [radius=.025];
\draw [fill] (a11) circle [radius=.025];
\draw [fill] (a12) circle [radius=.025];
\node at (-0.53, 1.02) {$\textbf{5}$}; 
\node at (-0.95, 0.65) {$\textbf{4}$}; 
\node at (-1.08, -0.41) {$\textbf{2}$}; 
\node at (1.14, 0.14) {$\textbf{3}$}; 
\node at (0.95, 0.65) {$\textbf{1}$}; 
\end{tikzpicture}
\end{figure}
\end{center}
\smallskip

Now we prove a technical lemma.

\begin{lemma} \label{aux}
The set $A$ of injective mappings $f :[a] \to [a + b]$ satisfying the condition
$f(i)= f(j)+1 \Rightarrow i>j$ is in a natural bijection with the set $B$ of all mappings $g :[a] \to
[b] \cup \{ 0 \}$. In particular, $|A| = (b +1)^a$.
\end{lemma}

\begin{proof}
Given $f \in A$, define $g_f := \phi(f) \in B$ by setting $g_f (i) = |[f(i)] \setminus f([a])|$,
i.e. $g_f (i)$ is the number of elements of $[f(i)]$ which are not in the image of $f$.

Conversely, given $g \in B$ reconstruct $f_g := \phi^{-1} (g) \in A$ as follows.
For each $i \in [b] \cup \{ 0 \}$, the elements of the preimage $g^{-1} (i) \subseteq [a]$ are sent to consecutive
elements of $[a + b]$ in the increasing order. We start with $i=0$, then take $i=1$ and so on up to $i=b$, and after each step we skip one element, so that if $M = \max \{ g^{-1}(i) \}$ and $m = \min \{ g^{-1} (i +1) \}$ then $f_g(m)= f_g(M)+2$. Clearly, these two constructions are inverse to each other.
\end{proof}

\begin{theorem}\label{teor1.1}
For every integer $1\leq k< n$, the \emph{characteristic polynomial} of $\cA^{(k,n)}$ is
$$\chi(\cA^{(k,n)},q)=q(q-n)^{n-1}\,.$$
\end{theorem}
\begin{proof}
Let $X=(k,n)$ and let us count $|X_q|$ by counting the number of different labelings for which there does not exist $(i,j)$ with $1\leq i<j\leq n$ such that $i\in X$ and $x_i=x_j+1$, or such that $i\notin X$ and $x_1=x_j+i$.

First, we choose one position, $a_1\in\mathbb{F}_q$, out of $q$ possible positions, to place the element that labels $1$, $a_1$. In what follows we assume, without loss of generality, that $a_1=0$.

Second, we place the elements $a_{k+1}, \ldots, a_n \in \mathbb{F}_q$. These elements cannot be placed in positions $q-k, \ldots, q-1$ because $a_1 \neq a_j+i$, for all $i,j \in [n]$ such that $i \leq k$ and $i<j$, nor in position 0, already taken by $a_1$. In addition, elements between consecutive unoccupied positions must be placed in increasing order since $a_i \neq a_j + 1$, for all $i,j \in [n]$ such that $k < i < j$.  Suppose that $a_{k+1}, \ldots, a_n \in \mathbb{F}_q$ are already placed. Then, let
$b_1, \ldots, b_{q-n}=q-k$ be the first $q-n$ unoccupied positions, and consider, for every $i$ with $i>k$, \emph{the order  $f(i)$  of the first empty position after $a_i$}. In other words,
$$f(i)=\min\,\{j\mid b_j>a_i\}\,.$$
This defines a mapping $f\colon\{k+1,\dotsc,n\}\to[q-n]$. The set of such mappings is clearly in bijection with the set of possible placements of these elements.

Third, the number of ways of placing $a_k, a_{k-1}, \ldots, a_2$ in the remaining positions in such a way thay if $a_i$ and $a_j$ are placed consecutively then $a_i < a_j$ is, by Lemma \ref{aux}, $(q-n)^{k-1}$.

Consequently, $|X_q| = q (q-n)^{n-k} (q-n)^{k-1} = q(q-n)^{n-1}$.
\end{proof}

For example, consider the configuration for $p=13$ and $n=5$ represented above, with labeling
$\bx=(2,9,3,11,12)$, and suppose that $X=(2,5)=\{3,4\}$. Note that $a_1=2$. Placing as above $3$, $4$ and $5$ we obtain the first $p-n=8$ unoccupied positions, in clockwise order,
$$\langle b_1,\dotsc, b_8\rangle =\langle 4,5,6,7,8,9,10,0\rangle$$
and $f\colon\{3,4,5\}\to[8]$ is such that $f(3)=1$ and $f(4)=f(5)=8$.
Now, $\ell_2=6\in[8]$ since it occupies the sixth position (out of eight) not labeled by $a_i$, for $i\in\{1,3,4,5\}$.

In terms of the arrangements $\cA^{(k,n)}$, this means, by a celebrated result of Zaslavsky \cite{Zas}, that the number of \emph{regions}, $r(\cA^{(k,n)})$, and the number of \emph{relatively bounded regions}, $b(\cA^{(k,n)})$, do not depend on $1\leq k<n$, being
\begin{align*}
r(\cA^{(k,n)})&=(-1)^n\chi(-1)=(n+1)^{n-1}
\shortintertext{and}
b(\cA^{(k,n)})&=(-1)^{n-1}\chi(1)=(n-1)^{n-1}\,.
\end{align*}
Note that these results were known (and proven similarly \cite{Ardila}) for both the Shi and the Ish arrangements \cite{AR,LRW}, although not for the remaining arrangements of form $\cA^{(k,n)}$, which are, to the best of our knowledge, considered here for the first time.
We represent both the Shi and Ish arrangements in dimension $n=3$ on Figure~\ref{arrs}.

The \emph{Whitney polynomials} of $\Shi_3$ and $\Ish_3$ are not equal (in fact, not even the numbers of faces of dimensions 1 and 2 ---which we can count directly in Figure~\ref{arrs}---  are equal).
Whereas $w(\Shi_3,t,q)=6qt^2+3q(2q-5)t+q(q-3)^2$ (Cf. \cite[Theorem~6.5]{Athan}), by \cite[Theorem~6.3]{Athan}, $w(\Ish_3,t,q)=7qt^2+2q(3q-8)t+q(q-3)^2$.

Note that in both cases the
$16$ regions are labeled injectively with elements of $\{0,1,2\}^3$, which differ exactly in two such elements: whereas $201$ is used as a label in $\Shi_3$ but not in $\Ish_3$, $022$ is used in $\Ish_3$ but not in $\Shi_3$. We see this in detail in Example~\ref{expf}.

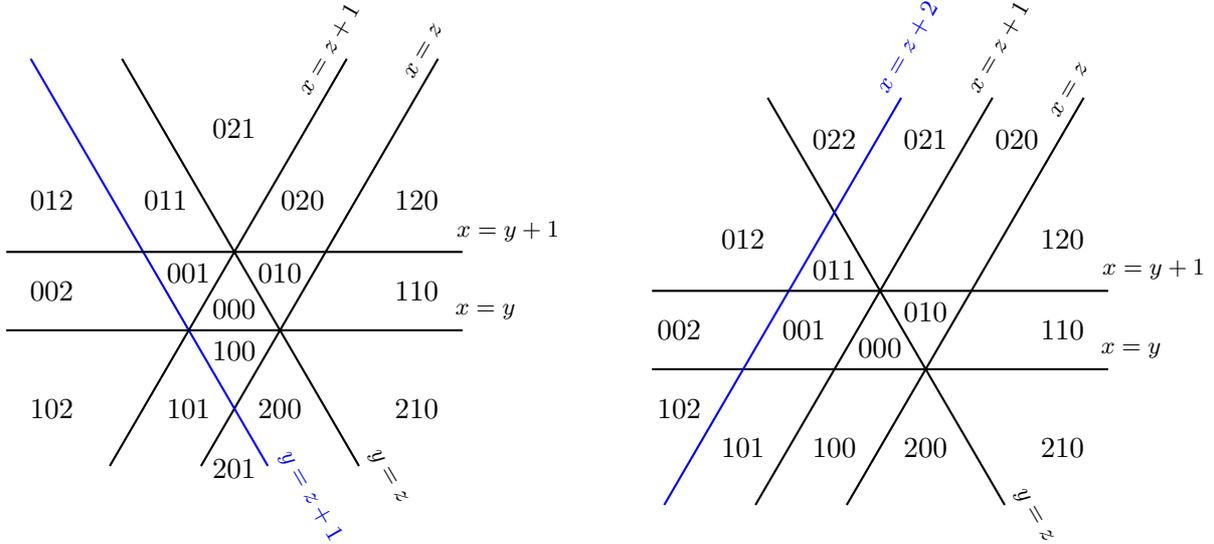
\begin{figure}[t]
\noindent
\begin{tikzpicture}[scale=1.2]
\draw  [thick]   (-0.866, -1.5) -- (1.732, 3.)  node [pos=1.,above,sloped]  {\scriptsize$x=z$};
\draw  [thick]   (-1.732, 3.) -- (0.866, -1.5)  node [pos=1.05,above,sloped]  {\scriptsize$y=z$};
\draw  [thick]  (-3., 0.) -- (2., 0.)   node [pos=1.05,above,sloped]  {\scriptsize$x=y$};
\draw  [thick,blue]  (-2.732, 3.) -- (-0.133, -1.5) node [pos=1.1,above,sloped]  {\scriptsize$y=z+1$};
\draw  [thick]  (-1.866, -1.5) -- (0.732,  3.) node [pos=1.,above,sloped]  {\scriptsize$x=z+1$};\small
\draw  [thick]  (-3., 0.866) -- (2., 0.866) node [pos=1.1,above,sloped]  {\scriptsize$x=y+1$};
\node at (-0.5,2.233){\small$021$};
\node at (0.25,1.433){\small$020$};
\node at (-2.5,1.433){\small$012$};
\node at (1.5,1.433){\small$120$};
\node at (-1.25,1.433){\small$011$};
\node at (-2.5,0.433){\small$002$};
\node at (-1.,0.633){\small$001$};
\node at (-0.5,0.233){\small$000$};
\node at (0.,0.633){\small$010$};
\node at (1.5,0.433){\small$110$};
\node at (-2.5,-0.866){\small$102$};
\node at (-0.5,-0.233){\small$100$};
\node at (-1.,-0.866){\small$101$};
\node at (0.,-0.866){\small$200$};
\node at (1.5,-0.866){\small$210$};
\node at (-0.5,-1.533){\small$201$};
\end{tikzpicture}
\hfill
\begin{tikzpicture}[scale=1.2]
\draw  [thick]  (-0.866, -1.5) -- (1.732, 3.)  node [pos=1.,above,sloped]  {\scriptsize$x=z$};
\draw  [thick]  (-1.732, 3.) -- (0.866, -1.5)  node [pos=1.05,above,sloped]  {\scriptsize$y=z$};
\draw  [thick]  (-3., 0.) -- (2., 0.) node [pos=1.05,above,sloped]  {\scriptsize$x=y$};
\draw  [thick,blue]  (-2.866, -1.5) -- (-0.268, 3.)  node [pos=1.1,above,sloped]  {\scriptsize$x=z+2$};
\draw  [thick]  (-1.866, -1.5) -- (0.732,  3.)  node [pos=1.1,above,sloped]  {\scriptsize$x=z+1$};
\draw  [thick]  (-3., 0.866) -- (2., 0.866) node [pos=1.1,above,sloped]  {\scriptsize$x=y+1$};
\node at (-1.,2.533) {\small$022$};
\node at (0.,2.533){\small$021$};
\node at (1.,2.533){\small$020$};
\node at (-2.,1.433){\small$012$};
\node at (1.5,1.433){\small$120$};
\node at (-1.,1.1){\small$011$};
\node at (-2.7,0.433){\small$002$};
\node at (-1.333,0.433){\small$001$};
\node at (-0.5,0.233){\small$000$};
\node at (0.,0.633){\small$010$};
\node at (1.5,0.433){\small$110$};
\node at (-2.7,-0.433){\small$102$};
\node at (-2.,-0.866){\small$101$};
\node at (-1.,-0.866){\small$100$};
\node at (0.,-0.866){\small$200$};
\node at (1.5,-0.866){\small$210$};
\end{tikzpicture}
\caption{Pak-Stanley labelings of $\Shi_3$ and $\Ish_3$}
\label{arrs}
\end{figure}

\section{The Pak-Stanley labeling}\label{sec.2}

Here, we consider a general process of labeling the regions of these arrangements. It was introduced by Pak and Stanley \cite{Stan2} for the Shi arrangement and was recently studied (in a general setting that covers our arrangements) by Mazin \cite{Mazin}, who calls it a \emph{Pak-Stanley}-labeling of the arrangement.

For our purposes, this labeling is characterized as follows. Let $\ell(\cR)$ be the label of a region $\cR$ and
write as usual $\be_i$ for the element of $\{0,1,\dotsc,n-1\}^n$ that differs from $\bz=(0,\dotsc,0)$ on the $i$th coordinate, which is equal to $1$. Let $\cR_0$ be the region defined by
$$x_n+1>x_1>x_2>\dotsb>x_n$$
(limited by the hyperplane $I_{1n}=S_{1n}$ and by the hyperplanes of form $C_{j\,j+1}$ for $1\leq j<n$).
Then set $\ell(\cR_0)=\bz$. Now, given two regions, $\cR_1$ and $\cR_2$, separated by a hyperplane $H$ such that $\cR_0$ and
$\cR_1$ are on the same side of $H$, let
$$\ell(\cR_2)=\ell(\cR_1)+\begin{cases}\be_i,&\text{if $H\in C_{ij}$;}\\\be_j,&\text{if $H\in S_{ij}\cup I_{ij}$.}\end{cases}$$

Define, for $X\subseteq(1,n)$, the directed multigraph $\cG^X=(V,A^X)$ such that $V=[n]$ and $A^X$ is the multiset (with $|\cA^X|$ elements) formed by adding, for each $H\in\cA^X$, a new arc $a\in A^X$, as follows. If $H=C_{ij}$ for some $1\leq i<j\leq n$, then $a=(i,j)$; if $H=S_{ij}$, then $a=(j,i)$; if $H=I_{ij}$, then $a=(j,1)$.
Hence, an arc of form $(j,1)$ may occur more than once in the multiset $A^X$.
In Figure~\ref{fig1} we consider the graphs $\cG^{(1,3)}$ and $\cG^{\varnothing}$ thus associated with $\Shi_3$ and $\Ish_3$, respectively.

Then, as shown by Mazin \cite{Mazin}, the set of labels obtained above is the set of \emph{$\cG^X$-parking functions}, according to the definition that follows.
In the case where $\cA^X$ is the Shi arrangement (i.e., when $X=(1,n)$),  the \emph{$\cG^X$-parking functions} are simply the \emph{parking functions}. In the case where $\cA^X$ is the Ish arrangement (when $X=\varnothing$),  the \emph{$\cG^X$-parking functions} are the \emph{Ish-parking functions}.

\begin{definition}[{Cf.~\cite{Mazin}}]\label{def-GPark}
Let $\cG=(V,A)$ be a (finite) directed connected multigraph without loops, where $V=[n]$ for some $n\in\N$.
A function $\pf\colon[n]\to\N_0$ is a \emph{$\cG$-parking function} if for every non-empty subset $I\subseteq[n]$ there exists a vertex $i\in I$ such that the number of elements $j\notin I$ such that $(i,j)\in A$, counted with multiplicity, is greater than or equal to $\pf(i)$.
\end{definition}

In the original Pak-Stanley labeling $\lambda$ of the regions of $\Shi_n$ (see \cite[p. 484]{Stan2}), also
$\lambda(\cR_0)=\bz$ but, given as before regions $\cR_1$ and $\cR_2$ separated by a hyperplane $H$ such that $\cR_0$ and $\cR_1$ are on the same side of $H$,
$$\lambda(\cR_2)=\lambda(\cR_1)+\begin{cases}\be_j,&\text{if $H\in C_{ij}$;}\\\be_i,&\text{if $H\in S_{ij}$.}\end{cases}$$
From Mazin's result, the set of labels is the same, since $\cG^{(1,n)}$ remains the \emph{complete directed graph on $n$ vertices} if all arcs are reversed.

This set of labels is the set $\PF_n$ of \emph{parking functions} \cite{Stan2}, which are the functions $\ba\colon[n]\to\{0,1,\dotsc,n-1\}$ such that
$$|\ba^{-1}(\{0,1,\dotsc,i\})|> i\qquad(i=0,1,\dotsc,n-1)\,.$$
In general, the graph $\cG^X$ associated with $\cA^X$ may be obtained from the complete directed graph on $n$ vertices by replacing every arc of form $(j,i)$ for $i<j\leq n$ and $i\in(1,n)\setminus X$ by an arc of form $(j,1)$.

\begin{figure}[t]
\noindent
\strut\hfill
\begin{tikzpicture}[scale=1.5]
\draw  [directed, thick]  (0,1) -- (.86,-.5);
\draw [fill] (0,1) node [above] {$1$} circle [radius=.05];
\draw  [directed, thick]  (0,1) -- (-.86,-.5) ;
\draw [fill] (-.86,-.5) node [below left] {$2$} circle [radius=.05];
\draw  [directed, thick]  (-.86,-.5) -- (.86,-.5) ;
\draw [fill] (.86,-.5) node [below right] {$3$} circle [radius=.05];
\draw  [directed, thick]  (.86,-.5) to [out=150,in=30] (-.86,-.5);
\draw  [directed, thick]  (-.86,-.5) to [out=30,in=270] (0,1);
\draw  [directed, thick]  (.86,-.5) to [out=150,in=270] (0,1);
\end{tikzpicture}
\hfill
\begin{tikzpicture}[scale=1.5]
\draw  [directed, thick]  (0,1) -- (.86,-.5);
\draw [fill] (0,1) node [above] {$1$} circle [radius=.05];
\draw  [directed, thick]  (0,1) -- (-.86,-.5) ;
\draw [fill] (-.86,-.5) node [below left] {$2$} circle [radius=.05];
\draw  [directed, thick]  (-.86,-.5) -- (.86,-.5) ;
\draw [fill] (.86,-.5) node [below right] {$3$} circle [radius=.05];
\draw  [directed, thick]  (-.86,-.5) to [out=30,in=270] (0,1);
\draw  [directed, thick]  (.86,-.5) to [out=150,in=270] (0,1);
\draw  [directed, thick]  (.86,-.5) to [out=135,in=285] (0,1);
\end{tikzpicture}
\hfill\strut

\noindent
\strut\hfill$\cG^{(1,3)}$\qquad\qquad\hfill$\cG^{\varnothing}$\hfill\strut
\caption{Directed multi-graphs associated with $\Shi_3$ ($\cG^{(1,3)}$) and $\Ish_3$ ($\cG^{\varnothing}$)}\label{fig1}
\end{figure}

\subsection{The $\dfs$-burning algorithm}

In this section  we extend the $\dfs$-burning algorithm 
by Perkinson, Yang and Yu \cite{PYY} and part of the results therein to directed multi-graphs $\cG=(V,A)$.

Let $\pf \colon [n] \to \N_0$ and $\bG=(\bV,\bA)$ where $\bV=\{0\}\cup V$ and $$\bA=\big\{(0,i)\mid i\in[n]\big\}\cup \big\{(j,i)\mid (i,j)\in A\big\}$$
and define, for each $v\in\bV$, the list $\cN(v)$ by ordering in some fixed way the \emph{multiset}
$$\big\{ i\in[n]\mid (v,i)\in\bA\big\}=\begin{cases}[n],&\text{ if $v=0$}\\
\big\{ i\in[n]\mid (i,v)\in A\big\},&\text{ if $v\neq0$}\end{cases}$$
Note that
 the algorithm terminates in all cases, either with \emph{all} vertices burned or else with only \emph{some} ---but not all--- vertices burned.
 In the first case, we say that \emph{$\pf$ fits $\cG$}.

{\small
\begin{algorithm}[ht]
\caption{ {\bf $\dfs$-burning algorithm (adapted)}.}
\begin{algorithmic}[1]
  \Statex
  \Statex{\hspace{-0.5cm}\sc algorithm}
  \Statex{\bf Input:} $\pf\colon[n]\to\N_0$
  \State $\bnv =\{0\}$
  \State $\damp = \{\,\}$
  \State $\tree= \{\,\}$
  \State execute {\sc dfs\_from}($0$)
  \Statex{\bf Output:} {$\bnv$, $\tree$ and $\damp$}
  \Statex \rule{0.5\textwidth}{.4pt}
  \Statex{\hspace{-0.5cm}\sc auxiliary function}
    \Function{\sc dfs\_from}{$i$}
      \ForAll{$j$ in $\cN(i)$}
	\If{$j\notin\bnv$}
	  \If{$\pf(j)=0$}
	    \State append $(i,j)$ to $\tree$
	    \State append $j$ to $\bnv$
	    \State {\sc dfs\_from}$(j)$
	  \Else
	     \State append $(i,j)$ to $\damp$
	     \State $\pf(j) = \pf(j)-1$
	  \EndIf
       \EndIf	
     \EndFor
   \EndFunction
  \end{algorithmic}
\label{dfs-burning}
\end{algorithm}
}

\begin{proposition}\label{algor}
Given a directed multigraph $\cG$ on $[n]$ and a function $\pf\colon[n]\to\N_0$,
$\pf$ fits $\cG$ if and only if $\pf$ is a $\cG$-parking function.
\end{proposition}
\begin{proof}
The arguments in this proof are essentially the same of Perkinson, Yang and Yu \cite{PYY}, but we present them for completeness sake.
We start by proving that if $\pf$ fits $\cG$ then $\pf$ is a $\cG$-parking function. Consider the lists $\bnv$, $\tree$ and
$\damp$ at the end of the execution of the algorithm and form the submultigraph $\cS$ of $\bG$ with the arcs from the multiset $M=\damp\cup\tree$.
Note that when, in the course of the algorithm, $(k,i)$ was inserted in $\damp$ or in $\tree$, $k$ was already in $\bnv$ but $i$ was not. Hence, $\cS$ is acyclic.

Consider also, for each $i\in[n]$, the multiset
$m(i)=\{ 0\leq k\leq n\mid (k,i)\in M\}$ and note that $\pf(i)=|m(i)|-1$, again counting the elements with multiplicity, and that $0$ may occur in $m(i)$ but not more than once.

Hence, if $I\subseteq[n]$, the restriction of $\cS$ to $I$ being acyclic,
there exists $i\in I$ such that $(k,i)\in M$ implies that $k\notin I$. Hence, the number of arcs of form $(k,i)$, with $k\notin I$, in $M$ (and so in $\bA$) is greater than or equal to $\pf(i)$.

Conversely, suppose that $\pf$ is a $\cG$-parking function and that, at the end of the execution of the $\dfs$-burning algorithm,
$\bnv=\langle0,v_1,\dotsc,v_k\rangle$, where $k<n$. Let $\pg$ be the value of $\pf$ at the termination of the algorithm, and note that $\pg(j)>0$ for every
$j\notin\{0,v_1,\dotsc,v_k\}$.
Let $I=[n]\setminus \{0,v_1,\dotsc,v_k\}$. Since $\pf$ is a $\cG$-parking function, there must exist
$j\in I$ such that there are $m \geq\pf(j)$ elements $i\in\{0,v_1,\dotsc,v_k\}$ with $(i,j)\in\bA$.
But this is impossible, since, by definition, $\pg(j)=\pf(j)-m$.
\end{proof}

Now we come back to the multiple digraphs $\cG^X$ associated with the hyperplane arrangements $\cA^X$ with $X\subseteq(1,n)$.

By definition, again (taking $I=[n]$), $0\in\pf([n])$. Since, for $i=1,\dotsc,n$, the number of arcs in $\bA$
with head $i$, counted with multiplicity, is always $n$,
if $\pf$ is a $\cG^X$-parking function for $X\subseteq(1,n)$, then
\begin{equation}\label{fcor}
\pf([n])\subseteq\{0,1,\dotsc,n-1\}\,.
\end{equation}

\begin{example} \label{expf}
Consider the multiple digraphs $\cG^{(1,3)}$ and $\cG^{\varnothing}$ represented in Figure~\ref{fig1}. We represent
$\overline{\cG^{(1,3)}}$ and $\overline{\cG^{\varnothing}}$
through the list
$\big\langle\cN(0),\dotsc,\cN(3)\big\rangle$, respectively
\begin{align*}
&\big\langle\langle3,2,1\rangle,\langle3,2\rangle,\langle3,1\rangle,\langle2,1\rangle\big\rangle\qquad\text{and}\\
&\big\langle\langle3,2,1\rangle,\langle3,3,2\rangle,\langle1\rangle,\langle2,1\rangle\big\rangle
\end{align*}

If we apply the algorithm with the graph $\overline{\cG^{(1,3)}}$ to $\pf_1=(2,0,1)$ and to $\pf_2=(0,2,2)$, we conclude that $\pf_1$ is a $\cG^{(1,3)}$-parking function whereas $\pf_2$ is not. With the graph $\overline{\cG^{\varnothing}}$, the situation is the opposite one.

In fact, for $\overline{\cG^{(1,3)}}$, if we take as input $\pf=\pf_1$ we obtain $\bnv=\langle0,2,3,1\rangle$ whereas if we take as input $\pf=\pf_2$ we obtain $\bnv=\langle0,1\rangle$. These lists, for $\overline{\cG^{\varnothing}}$, are respectively $\langle0,2\rangle$ and $\langle0,1,3,2\rangle$.

We may follow the execution of the algorithm in a drawing where, for each $i\in[3]$, there are $\pf_i+1$ empty boxes;
during the execution of {\sc dfs\_from}{($i$)}, for a given value of $j$, at Line~14, the upmost empty box in column $j$ is filled with $i$ . Then, the arcs of form $(k,i)$ in $\damp$ are marked $k$ in column $i$ \emph{above} the bottom row, and the bottom row is filled with $k$ for $(k,i)\in\tree$.
We obtain, respectively\footnote{The two arcs of form $(1,3)$ appear, one in $\damp$ and the other one in $\tree$, when the $\dfs$-burning algorithm
is executed with $\pf=\pf_2$ and graph $\cG^{\varnothing}$.},
$$
\begin{array}{ccc}
\cline{1-1}
\cx{3}&&\\
\cline{1-1}\cline{3-3}
\cx{2}&&\cx{0}\\
\hline
\hline
\cx{0}&\cx{0}&\cx{2}\\
\hline
{\scriptstyle1}&{\scriptstyle2}&{\scriptstyle3}
\end{array}\quad
\begin{array}{ccc}
\cline{2-3}
&\cx{0}&\cx{0}\\
\cline{2-3}
&\cx{1}&\cx{1}\\
\cline{2-3}\\[-12pt]
\hline
\cx{0}&\cx{}&\cx{}\\
\hline
{\scriptstyle1}&{\scriptstyle2}&{\scriptstyle3}
\end{array}
\qquad\qquad\qquad
\begin{array}{ccc}
\cline{1-1}
\cx{2}&&\\
\cline{1-1}\cline{3-3}
\cx{0}&&\cx{0}\\
\hline
\hline
\cx{}&\cx{0}&\cx{}\\
\hline
{\scriptstyle1}&{\scriptstyle2}&{\scriptstyle3}
\end{array}\quad
\begin{array}{ccc}
\cline{2-3}
&\cx{0}&\cx{0}\\
\cline{2-3}
&\cx{1}&\cx{1}\\
\cline{2-3}\\[-12pt]
\hline
\cx{0}&\cx{3}&\cx{1}\\
\hline
{\scriptstyle1}&{\scriptstyle2}&{\scriptstyle3}
\end{array}
$$
\end{example}

\begin{remark}
We use the same notation, $(i,j)$, for all of the different arcs connecting vertex $i$ to vertex $j$ (following  Mazin \cite{Mazin})
so as to keep the  description of Perkinson, Yang and Yu \cite{PYY} of their algorithm. But note that
$\tree$ must distinguish these arcs from each other.
One way of doing this is to distinguish the various occurrences of a vertex in the list of neighbors of a vertex, e.g., taking for $\Ish_3$
$$\big\langle\langle3,2,1\rangle,\langle3_1,3_2,2\rangle,\langle1\rangle,\langle2,1\rangle\big\rangle\,.$$
and changing Line~6 to ``\textbf{foreach} {$j_\ell$ in $\cN(i)$}''
\footnote{In an actual implementation with few points, we can take, if $n<10$, for example, $j_\ell=10\ell+j=k$ so that $j=\textrm{Mod}(k,10)$.}
 and Line~9 to ``append $(i,j_\ell)$ to $\tree$'', but keeping Line~7, ``\textbf{If} {$j\notin\bnv$}'', etc. We then obtain at the end, for $\pf=022$,
 \begin{align*}&\tree=\langle (0,1),(1,3_2),(3,2)\rangle
 \shortintertext{but for $\pf=021$ we obtain}
&\tree=\langle (0,1),(1,3_1),(3,2)\rangle\,.
 \end{align*}
Note also that the order within the lists $\cN(i)$ is not relevant for the result in itself, but $\tree$ and the submultigraph $\cS$ depend on it.
\end{remark}

{\small
\begin{algorithm}[ht]
\caption{ {\bf Tree to parking function algorithm (adapted)}.}\label{inverse}
\begin{algorithmic}[1]
  \Statex
  \Statex{\hspace{-0.5cm}\sc algorithm}
  \Statex{\bf Input:} Spanning tree $T$ rooted at $r$ with edges directed away
  from root.
  \State $\bnv =\{r\}$
  \State $\damp = \{\,\}$
  \State $\pf=0$
  \State execute {\sc tree\_from$(r)$}
  \Statex{\bf Output:} {$\pf\colon V\setminus\{r\}\to\N$}
  \Statex \rule{0.5\textwidth}{.4pt}
  \Statex{\hspace{-0.5cm}\sc auxiliary function}
  \Function{\sc tree\_from}{$i$}
      \ForAll{$j$ in $\cN(i)$}
      \If{$j\notin\bnv$}
	\If{$(i,j)$ is an edge of $T$}
	    \State append $j$ to $\bnv$
	    \State {\sc tree\_from}$(j)$
	 \Else
	   \State $\pf(j) = \pf(j)+1$
	   \State append $(i,j)$ to $\damp$
	 \EndIf
      \EndIf
    \EndFor
  \EndFunction
  \end{algorithmic}
\end{algorithm}
}

\begin{proposition}\label{bij}
The $\cG$-parking functions are in bijection with the spanning \emph{arborescences} (the directed rooted trees with edges pointing away from the root) of $\cG$ that are rooted in $0$.

\end{proposition}
\begin{proof}
It is easy to see that $\tree$ defines a unique arborescence in all cases. For the inverse, following Perkinson, Yang and Yu \cite{PYY}, consider the Algorithm~\ref{inverse}.
\end{proof}

Define as usual  the Laplacian matrix $L=L(\cG)$ by
$$L_{ij}=\begin{cases} -m_{ij},& \text{if $i\neq j$}\\
\mathrm{outdeg}(i)-m_{ii},&\text{if $i=j$}
\end{cases}$$
where $m_{ij}$ is the number of arcs of form $(i,j)$ and $\mathrm{outdeg}(i)$ is the number of arcs with head $i$.

\begin{corollary}\label{cor}
Given a directed multigraph $\cG$ on $[n]$, the number of $\cG$-parking functions is the determinant of the matrix $L_0$ obtained from the Laplacian $L=L(\cG)$
by deleting the row and the column corresponding to $0$.
\end{corollary}
\begin{proof}
This is a direct consequence of Proposition~\ref{bij} and of Kirchhoff's Matrix Tree Theorem for directed multi-graphs (Cf. \cite[Theorem 5.6.4.]{Stanv2}, for example).
\end{proof}

\begin{theorem}\label{bij2}
For every natural $n\geq3$ and $1\leq k< n$,
the Pak-Stanley labeling considered above defines a bijection between the regions of $\cA^{(k,n)}$ and the $\cG^{(k,n)}$-parking functions.
\end{theorem}
\begin{proof}
The surjectivity of the labeling was established by Mazin \cite[Theorem 3.1]{Mazin}; by Corollary~\ref{cor} all we have to do is to show that the determinant of the matrix $\cM^{(k,n)}$, obtained from the Laplacian of $\overline{\cG^{(k,n)}}$ by deleting the row and the column corresponding to $0$, is the number $(n+1)^{n-1}$ of regions of $\cA^{(k,n)}$, according to Theorem~\ref{teor1.1}.
We now show that all matrices of form $\cM^X$ with $X\subseteq(1,n)$ may be obtained from $\cM^{(1,n)}$ (corresponding to $\Shi_n$) by multiplication on the right by a unitary matrix. 

Suppose that $(1,\ell+1)\subseteq X\subseteq(1,n)$ and so
the entries of the first $\ell$ columns of $\cM^X$ 
are either $n$, if they belong to the main diagonal of $\cM^X$, or $-1$. Note that
$\cM^X$ and $\cM^{X\setminus\{\ell\}}$  differ in that, from the first  to the second matrices, in column $j>\ell$
the entry of the first line is decreased by $1$ whereas  the entry of line $\ell$  is increased by $1$. 
Hence,  being $M_\ell=\big(m_{ij}\big)_{1\leq i,j\leq n}$ such that
\begin{align*}&m_{ij}=\begin{cases} -\frac{1}{n+1},&\text{if $i=1$ and $j>\ell$;}\\\frac{1}{n+1},&\text{if $i=\ell$ and $j>\ell$;}\\0,&\text{otherwise,}\end{cases}\\[10pt]
&\cM^{X\setminus\{\ell\}}=\cM^X\cdot (I_n+M_\ell)
\end{align*}
where as usual $I_n$ is the identity matrix of order $n$. Note also that $A\cdot M_\ell=\bz$ if $A=\big(a_{ij}\big)_{1\leq i,j\leq n}$ is such that $a_{ij}=0$ whenever
$1\leq i,j\leq\ell$.
Thus, if $X=(1,n)\setminus\{x_1,\dotsc,x_k\}$ with $x_1>\dotsb>x_k$,
\begin{align*}
\cM^X&=\cM^{(1,n)}\cdot (I_n+M_{x_1})\dotsb(I_n+M_{x_k})\\
&=\cM^{(1,n)}\cdot (I_n+M_{x_1}+\dotsb+M_{x_k})
\end{align*}
the last equality since  $M_{x_i}M_{x_j}=0$ if $i>j$ and hence $M_{x_{i_1}}M_{x_{i_2}}\!\dotsb\!M_{x_{i_m}}=0$ for $i_1>i_2>\dotsb>i_m$.
In particular, $$\det(\cM^X)=\det(\cM^{(1,n)})=(n+1)^{n-1}\,.\qedhere$$
\end{proof}

\begin{example}
Going back to the Shi and Ish arrangements in dimension $3$, we have the following Laplacians.
$$L\big(\cG^{(0,n)}\big)=\text{\footnotesize
$\bordermatrix{~&0&1&2&3\cr0&0&-1&-1&-1\cr1&0&3&-1&-1\cr2&0&-1&3&-1\cr3&0&-1&-1&3\cr}$}\ ,\ 
L\big(\cG^\varnothing\big)=\text{\footnotesize
$\bordermatrix{~&0&1&2&3\cr0&0&-1&-1&-1\cr1&0&3&-1&-2\cr2&0&-1&3&0\cr3&0&-1&-1&3\cr}$}\,.$$
Note that $
\left(\!\begin{smallmatrix}\stt3&-1&-2\\\stt-1&3&0\\\stt-1&-1&3\end{smallmatrix}\!\right)=
\left(\!\begin{smallmatrix}\stt3&-1&-1\\\stt-1&3&-1\\\stt-1&-1&3\end{smallmatrix}\!\right)\cdot
\left(\!\begin{smallmatrix}\stt1&0&-1/4\\\stt0&1&1/4\\\stt0&0&1\end{smallmatrix}\!\!\right)$\,.
\end{example}

\begin{remark}\label{krat}
For an alternative proof, note that, for every $X\subseteq(1,n)$, $1$ is the sum of the entries of any column of $\cM^X$ and the elements below the main diagonal are all equal to $-1$. Hence, by adding up all rows in the first row, by subtracting each column to the column ahead and by expanding the determinant with respect to the first row (now reduced to $10\dotsb0$) we obtain the determinant of a triangular matrix with $n-1$ diagonal entries equal to $n+1$.
\end{remark}

We may now characterize the labels of the Pak-Stanley labeling of the Ish arrangement (the Ish-parking functions), again based on the $\dfs$-burning algorithm.

\begin{definition}\label{def.crc}
Let $\ba=(a_1,\dotsc,a_n)\in\N_0^n$.  The \emph{reverse center of $\ba$},
$\tZ(\ba)$, is the largest subset $X = \{ x_1, \ldots, x_\ell \}$ of $[n]$  with $\ n\geq x_1 > \cdots > x_\ell\geq1\ $ 
with the property that $a_{x_i} < i$ for every $i \in [\ell]$.
Note that if this property holds for both  $X,Y\subseteq[n]$, then it holds for $X\cup Y$, and so this concept is well-defined.
\end{definition}

\begin{lemma}\label{Ishlabels}
Let $\bnv=\langle v_0\!=\!0,v_1,\dotsc,v_k\rangle$ be the list defined by the $\dfs$-burning algorithm
with the lists $\cN(i)$ sorted in descending order for $i=0,\dotsc,n$,
applied to $\cG^\varnothing$ with input $\pf\colon[n]\to\{0,1,\dotsc,n-1\}$, at the end of the execution.
Then the following statements are equivalent:
\begin{enumerate}
\setlength{\itemindent}{1em}
\item[(\thetheorem.1)] For some $1\leq\ell\leq k$, $v_\ell=1$;
\item[(\thetheorem.2)] $1\in\tZ(\ba)$;
\item[(\thetheorem.3)] as a set, $\bnv=\{0\}\cup[n]$.
\end{enumerate}
\end{lemma}
\begin{proof}$\ $\\[-18pt]
\begin{enumerate}
{\setlength{\itemindent}{1em}
\item[(\thetheorem.1)]$\!\implies$(\thetheorem.2). Note that $j<i$ if $j\in\cN(i)$ and $i\neq 0,1$. Hence, $v_1>v_2>\dotsb>v_{\ell-1}$ and $a(v_{i_m})+1\leq m$ because both the dampened edges and the tree edge leading to $v_{i_m}$ must have origin $v_{i_p}$ with $0\leq p<m$.
\item[(\thetheorem.2)]$\!\implies$(\thetheorem.3). Suppose that $1$ belongs to the reverse center of $\pf$ but there is a greatest element $j\in[n]$ not in $\bnv$. Then, during the execution of the algorithm (more precisely, during the execution of Line~14)
the value of $\pf(j)$ has decreased once for $i=0$ (that is, as a neighbor of $0$), once for each value of $i>j$ (in a total of $n-j$), and $j-1$ times for $i=1$, and is still greater than zero. Hence $\pf(j)>n$, a contradiction with \eqref{fcor}.
\item[(\thetheorem.3)]$\!\implies$(\thetheorem.1). Obvious.\qedhere}
\end{enumerate}
\end{proof}

\begin{proposition}\label{prop.ish}
A function $\ba\colon[n]\to\N_0$ is an \emph{Ish-parking function} if and only if
$1$ belongs to the reverse center of $\pf$.
\end{proposition}
\begin{proof} Follows immediately from Lemma~\ref{Ishlabels}.\end{proof}

We write $\IPF_n$ for the set of Ish-parking functions of length $n$ and $\IPF^k_n$ for the set of those of reverse center of size $k$.
\newcommand{\RW}{\mathsf{RW}}
\newcommand{\bas}{\mathbf{a^*}}

\section{The number of Ish-parking functions with reverse center of a given length}

In this section we prove that there are as many Ish-parking functions with \emph{reverse center of a given length} as parking functions
with \emph{center} of the same length, as defined in a previous paper of ours \cite{DGO}.
Recall that \emph{centers} occur in relation with the original Pak-Stanley labeling of the Shi arrangement, namely for the determination of the region labeled with a given parking function \cite{DGO}. See the Appendix for more details.
This number was studied in another article \cite{DGO2}. We briefly recall here some notation and some results thereof.

Throughout this section, $\ba=a_1\dotsb a_n\in\{0,1,\dotsc,n-1\}^n$ is a generic element;
denote the length of the reverse center of $\ba$ by $\tz(\ba):=|\tZ(\ba)|$,
let $\bas=\ba+1\dotsb1$
and $\PF^*_n=\big\{\bas\mid\ba\in\PF_n\big\}$.
The set $\PF^*_n$ is the set of parking functions as defined on our previous papers \cite{DGO,DGO2}, where we defined $z(\bas)$ as the length of the center $Z(\bas)$. We recall that, for any $a \in [n]^n$, the center of $\ba$ is the largest subset $X = \{ x_1, \ldots, x_\ell \}$ of $[n]$ such that $1 \leq x_1 < \cdots < x_\ell \leq n$ and $a_{x_i} \leq i$ for every $i \in [\ell]$.  Then, by definition, $z(\bas)=\tz(a_n\dotsb a_1)$.
We consider the enumerators
$$\ZP_n(t)=\sum_{\ba\in\PF_n} t^{z(\bas)}\ ,\quad\ZI_n(t)=\sum_{\ba\in\IPF_n} t^{\tz(\ba)}\,.$$

\begin{theorem}[{\cite[Cf. Theorems 2.1 and 5.1]{DGO2}}]\label{conta}
For every integers $1\leq r\leq n$,
\begin{align*}
[t^r]\big(\ZP_n(t)\big)&= r! \sum_{i_1 + \dotsb + i_r = n-r}
(n-1)^{i_1} (n-2)^{i_2} \dotsb (n-r)^{i_r}\\ &= r\ \sum_{j=0}^{r-1}
(-1)^j \binom{r-1}{j} (n-1-j)^{n-1}\,.
\end{align*}
\end{theorem}

For a fixed subset $A = \{ i_1, i_2, \ldots, i_{k-1} \}\subseteq[n]$ with $i_1 > i_2 > \cdots > i_{k-1}>1$, let $i_k=1$ and $i_0=n+1$
and note that, by Proposition~\ref{prop.ish}, \emph{$\ba$ is an Ish-parking function with reverse center $A\cup\{1\}$}
(and  $\tz(\ba)=k$)  if and only if
\begin{itemize}
\item $(a_{i_1},a_{i_2},\dotsc,a_{i_k}\!\!=\!\!a_1) \in \{0\} \times\{0,1\}\times
  \cdots \times\{0,\dotsc,k-1\}$,
\item $a_j \in\{\ell,\dotsc,n-1\}$ if $i_{\ell} < j < i_{\ell-1}$ with $1\leq\ell \leq k$.
\end{itemize}
Graphically,
{$$\begin{array}{cccccccc}
\fbox{$a_{i_k} {\scriptstyle \leq k-1}$} &
{\neq
\begin{cases}\scriptstyle\ 0 \\[-5pt]\scriptstyle\ 1 \\[-5pt]\ \vdots \\[-5pt] \scriptstyle k-1 \end{cases}} &
&
{\scriptstyle \neq
\begin{cases} \scriptstyle 0 \\[-5pt] \scriptstyle 1 \\[-5pt] \scriptstyle 2 \end{cases}} &
\fbox{$a_{i_2} {\scriptstyle \leq 1}$} &
{\scriptstyle\neq
\begin{cases} \scriptstyle 0 \\[-5pt] \scriptstyle 1 \end{cases}} &
\fbox{$a_{i_1} {\scriptstyle \leq 0}$} &
{\scriptstyle \neq 0}\\[-5pt]
\underbracket[.5pt]{\hspace{.8cm}}_{i_k=1} &
\underline{\hspace{1.2cm}} &
 \ldots & \underline{\hspace{1.2cm}} &
\underbracket[.5pt]{\hspace{.8cm}}_{i_2} &
 \underline{\hspace{1.2cm}} &
\underbracket[.5pt]{\hspace{.8cm}}_{i_1} & \underline{\hspace{1.2cm}}
\end{array}$$}
Hence, the number of Ish-parking functions of length $n$ and reverse center of length $k$ is
$$|\IPF^k_n|=k! \sum_{\scriptscriptstyle j_1+j_2 + \dotsb + j_k=n-k} (n-1)^{j_1} (n-2)^{j_2}\dotsb (n-k)^{j_k}$$
and, by Theorem~\ref{conta},
$$\ZP_n(t)=\sum_{\ba\in\PF_n} t^{z(\bas)}=\sum_{\ba\in\PF_n} t^{\tz(\ba)}=\sum_{\ba\in\IPF_n} t^{\tz(\ba)}=\ZI_n(t) \nonumber\,,$$
the second equality since $a_1\dotsb a_n\in\PF_n$ if and only if $a_n\dotsb a_1\in\PF_n$.
Hence,
\begin{equation}
\begin{split}
|\IPF^k_n|
&= k\ \sum_{j=0}^{r-1} (-1)^j \binom{k-1}{j} (n-1-j)^{n-1}\\
&= k\ \big|\big\{f\colon[n-1]\to[n-1]\mid [k-1]\subseteq f([n-1])\big\}\big|\,.
\end{split}\label{Whitney}
\end{equation}

\begin{remark}
We have actually first obtained the number of parking functions with center of length $\ell$ by counting \emph{rook words with run $\ell$}, for $1\leq\ell\leq n$ \cite{DGO2}.
Rook words were introduced by Leven, Rhoades and Wilson \cite{LRW} also as labels of the regions of the Ish arrangements. If $\bb=b_1,\dotsb b_n\in[n]^n$, the \emph{run} of $\bb$ is $\run(\bb) = \max\big\{i\in[n]\mid [i]\subseteq
\{b_1,\dotsc,b_n\}\big\}$
and $\bb$ is a \emph{rook word} if $b_1\leq\run(\bb)$ \cite{LRW}. Let
$$R (\bb) = \big\{ \max \ \bb^{-1} (\{ j \}) \mid 1 \leq j \leq \run (\bb) \big\}\,.$$
Then \eqref{Whitney} is an easy consequence of the following result.
\begin{theorem}\cite[Theorem 3.4]{DGO2}
There exist $\Psi,\Phi:[n]^n$, inverse to each other, such that:
\begin{enumerate}
\item For every $\bb \in [n]^n$,\\
$z(\bb) = \run (\Phi (\bb))$, $Z(\bb) = R (\Phi (\bb))$, $\run (\bb) = z(\Psi (\bb))$, and $R (\bb) =Z(\Psi (\bb))$;
\item $\Phi (\PF^*_n) = \PF^*_n$.
\end{enumerate}
\end{theorem}
Now, for counting parking functions with a given run $k$, we may count \emph{rook words}  with the same run, as we have proved. Although we cannot use the same argument for the hyperplane arrangements between Shi and Ish, we have verified the following conjecture up to $n=8$.
\end{remark}

\begin{conjecture}\label{conj}
For every $n\in\N$, $n\geq3$,  the number of $\cG^X$-parking functions with reverse center of a given length does not depend on $X\subseteq(1,n)$.
\end{conjecture}

\begin{remark}
Equation~\ref{Whitney} shows that the number of parking functions with center (or reverse center) of length $r$ is related with the number of relatively bounded faces of dimension $d$, $0\leq d<n$, $g_d=(-1)^{d+1}\;[t^{n-d}]\,w(\Shi_n,t,1)$, where $w(\Shi_n,t,q)$ is the Whitney polynomial of the Shi arrangement \cite[Theorem 6.5]{Athan}. The relation is as follows.
$$|\PF^k_n|=k\,g_{n-k+1} /\textstyle\binom{n}{k-1}\,.$$
For example, if $n=3$ there are $6$ relatively bounded faces of dimension $1$, $9$ relatively bounded faces of dimension $2$ and $4$ relatively bounded faces of dimension $3$; see Figure~\ref{arrs}. And there are $6=3\times6/\binom{3}{2}$ parking functions with reverse center of length $3$
{\footnotesize ($000$, $010$, $110$, $210$, $200$, and $100$)},
$6=2\times9/\binom{3}{1}$ with reverse center of length $2$
{\footnotesize ($110$, $120$, $001$, $002$, $102$, and $101$)},
and $4=1\times4/\binom{3}{0}$ with reverse center of length $1$ {\footnotesize ($011$, $012$, $021$, and $201$)}. But the same does not happen with the Ish arrangement, where the Whitney polynomial is different, since the number of Ish-parking functions with center of a given length is equal to the corresponding number of parking functions.
\end{remark}

\appendix\section{The inverse of the labeling of the Shi arrangement}

By definition, the inverse of a (bijective) labeling $\mu$ of the regions of an arrangement maps each label $\ba$ to the region $\cR$ such that $\ba=\mu(\cR)$. In a previous paper \cite{DGO}, we gave a recursive definition of the inverse of the original Pak-Stanley labeling $\lambda$ of the regions of the Shi arrangement. Let us extend this definition to the labeling $\ell$.

We begin by considering how the two labelings are related to each other.

Every region $\cR$ of $\Shi_n$ corresponds bijectively to a pair
$(\bw,\fI)$, called a \emph{valid pair} \cite{Stan2}, where
\begin{itemize}
\item $\bw=(w_1,\dotsc,w_n)\in\mathfrak{S}_n$;
\item $\fI$ is an \emph{anti-chain of proper intervals}, meaning that
  $\fI$ is a collection of intervals $[i,j]$ with $1\leq i<j\leq n$
  such that if $I,I'\in\fI$ and $I\neq I'$ then $I\nsubseteq I'$ (and
  $I'\nsubseteq I$);
\item for every $I=[b,e]\in\fI$, $w_b<w_e$.
\end{itemize}
Now, consider the labels of $\cR$ (that we see as labels of $(\bw,\fI)$) by the two labelings of the regions of the Shi arrangement of Section~\ref{sec.2}, the original Pak-Stanley labeling $\lambda$ and the labeling $\ell$,
\begin{align*}
\lambda(\bw,\fI)&=\big(\lambda(\bw,\fI,1),\dotsc,\lambda(\bw,\fI,n)\big)\in\PF_n\\
\ell(\bw,\fI)&=\big(\ell(\bw,\fI,1),\dotsc,\ell(\bw,\fI,n)\big)\in\PF_n
\end{align*}
By definition, the region $\cR$ corresponding to the valid pair $(\bw,\fI)$ is separated from the region
$\cR_0$ by the hyperplane of equation $x_i=x_j$ for each $(i,j)$ such that $1\leq i<j\leq n$ and $w_i>w_j$, and
by the hyperplane of equation $x_i=x_j+1$ for each $(i,j)$ such that $1\leq i<j\leq n$, $w_i<w_j$ and it does not exist $I\in\fI$ with both $i$ and $j$ in $I$ \cite{Stan2}. Hence,
\begin{align*}
\lambda(w,\fI,w_i)&=\big|\big\{j\in[n]\mid i<j \,\wedge\, w_i>w_j \big\}\big|\\
&{}+{}\big|\big\{j\in[n]\mid i<j \,\wedge\, w_i<w_j\,,\ \text{no $I\in\fI$ satisfies $i,j\in I$\,}\big\}\big|\\
&=n-i-\big|\big\{j\in[n]\mid i<j\leq e_i \,\wedge\, w_i<w_j \big\}\big|\\
\ell(w,\fI,w_j) &=\big|\big\{i\in[n]\mid i<j \,\wedge\, w_i>w_j \big\}\big|\\
&{}+{}\big|\big\{i\in[n]\mid i<j \,\wedge\, w_i<w_j\,,\ \text{no $I\in\fI$ satisfies $i,j\in I$\,}\big\}\big|\\
&=j-1-\big|\big\{i\in[n]\mid b_j\leq i<j \,\wedge\, w_i<w_j \big\}\big|
\end{align*}
where $e_i=\max\big\{e\mid i\in[b,e]\in\fI \big\}$ and $b_j=\min\big\{b\mid j\in[b,e]\in\fI \big\}$.

Given the valid pair $P=(\bw,\fI)$ with $\bw=(w_1,w_2,\dotsc,w_n)$, consider the new (valid) pair
$\widetilde{P}=(\tbw,\tI)$, where
\begin{align*}
&\tbw=(n+1-w_n,n+1-w_{n-1},\dotsc,n+1-w_1)\,,\\
&\tI=\big\{[n+1-e,n+1-b]\mid [b,e]\in\mathcal{I}\big\}
\end{align*}
and note that $\lambda(\tbw,\tI,\tw_{n+1-j})=\ell(\bw,\fI,w_j)$. Hence, since $\tw_{n+1-j}=n+1-w_j$,
we obtain the following result.

\begin{lemma}\label{lem.inv}
With the previous notations,
$\lambda(\tbw,\tI)$ is the reverse of $\ell(\bw,\fI)$.\qed
\end{lemma}

It is now clear how to obtain $\ell^{-1}$ from $\lambda^{-1}$. Let us apply this method to an example taken from Stanley
\cite[Example, p.484]{Stan2}.
Let $\cR$ be the region of the Shi arrangement in $\Reals^9$ defined by:
{\small \begin{align*}
&x_5>x_2>x_1>x_7>x_6>x_9>x_3>x_4>x_8\,,\\
& x_7+1>x_5\,,x_3+1>x_2\,,x_8+1>x_7\,,\\
& x_5>x_6+1\,,x_1>x_4+1\,.
\end{align*}}
The valid pair associated with $\cR$ is $(\bw,\fI)$ where $\bw=5\,2\,1\,7\,6\,9\,3\,4\,8\in\mathfrak{S}_9$ and $\fI=\big\{[1,4],[2,7],[4,9]\big\}$, and we may represent it simply by \footnote{Note that the same valid pair was represented
$\begin{tikzpicture}[scale=.2]
\draw (0,0) node [inner sep=.25mm,minimum size=1mm] (y5){$8$};
\draw (1,0) node [inner sep=.25mm,minimum size=1mm] (y2){$4$};
\draw (2,0) node [inner sep=.25mm,minimum size=1mm] (y1){$3$};
\draw (3,0) node [inner sep=.25mm,minimum size=1mm] (y7){$9$};
\draw (4,0) node [inner sep=.25mm,minimum size=1mm] (y6){$6$};
\draw (5,0) node [inner sep=.25mm,minimum size=1mm] (y9){$7$};
\draw (6,0) node [inner sep=.25mm,minimum size=1mm] (y3){$1$};
\draw (7,0) node [inner sep=.25mm,minimum size=1mm] (y4){$2$};
\draw (8,0) node [inner sep=.25mm,minimum size=1mm] (y8){$5$};
\draw  [thick,looseness=0.75]  (y5) to [out=90,in=90] (y9);
\draw  [thick,looseness=0.75]  (y1) to [out=90,in=90] (y4);
\draw  [thick,looseness=0.75]  (y9) to [out=90,in=90] (y8);
\end{tikzpicture}$ in our previous article (Cf. \cite[Example~2.6]{DGO}).}\,:
\\[-15pt]
$$\begin{tikzpicture}[scale=.3]
\draw (0,0) node [inner sep=.5mm,minimum size=2mm] (x5){$5$};
\draw (1,0) node [inner sep=.5mm,minimum size=2mm] (x2){$2$};
\draw (2,0) node [inner sep=.5mm,minimum size=2mm] (x1){$1$};
\draw (3,0) node [inner sep=.5mm,minimum size=2mm] (x7){$7$};
\draw (4,0) node [inner sep=.5mm,minimum size=2mm] (x6){$6$};
\draw (5,0) node [inner sep=.5mm,minimum size=2mm] (x9){$9$};
\draw (6,0) node [inner sep=.5mm,minimum size=2mm] (x3){$3$};
\draw (7,0) node [inner sep=.5mm,minimum size=2mm] (x4){$4$};
\draw (8,0) node [inner sep=.5mm,minimum size=2mm] (x8){$8$};
\draw  [thick,looseness=0.75]  (x8) to [out=90,in=90] (x7);
\draw  [thick,looseness=0.75]  (x7) to [out=90,in=90] (x5);
\draw  [thick,looseness=0.75]  (x3) to [out=90,in=90] (x2);
\end{tikzpicture}\ .$$
Then, for instance, $b_5=2$ since this is the index of the leftmost element where an arc begins that ``covers'' $w_5=6$,
$\ \ell(w,\fI,6)=2=5-1-2$, being
$\big\{i\in[n]\mid 2\leq i<5 \,\wedge\, w_i<6 \big\}=\{2,3\}$.
All in all,
\begin{equation}\label{eq.inv}
\ba=\ell(\bw,\fI)=2\,1\,4\,6\,0\,2\,0\,4\,1\in\PF_9\,.
\end{equation}
Then $\tba=\lambda(\tw,\tI)$, that we may represent by
$\begin{tikzpicture}[scale=.25]
\draw (0,0) node [inner sep=.25mm,minimum size=1mm] (z5){$2$};
\draw (1,0) node [inner sep=.25mm,minimum size=1mm] (z2){$6$};
\draw (2,0) node [inner sep=.25mm,minimum size=1mm] (z1){$7$};
\draw (3,0) node [inner sep=.2mm,minimum size=.5mm] (z7){$1$};
\draw (4,0) node [inner sep=.2mm,minimum size=.5mm] (z6){$4$};
\draw (5,0) node [inner sep=.25mm,minimum size=1mm] (z9){$3$};
\draw (6,0) node [inner sep=.25mm,minimum size=1mm] (z3){$9$};
\draw (7,0) node [inner sep=.25mm,minimum size=1mm] (z4){$8$};
\draw (8,0) node [inner sep=.25mm,minimum size=1mm] (z8){$5$};
\draw  [thick,looseness=0.75]  (z5) to [out=90,in=90] (z9);
\draw  [thick,looseness=0.75]  (z1) to [out=90,in=90] (z4);
\draw  [thick,looseness=0.75]  (z9) to [out=90,in=90] (z8);
\end{tikzpicture}$, is $\tba=1\,4\,0\,2\,0\,6\,4\,1\,2\in\PF_9$.

If we want to recover $(\bw,\fI)$ out of $\ba$, we may start by reversing this parking function, obtaining again
$\tba=1\,4\,0\,2\,0\,6\,4\,1\,2\in\PF_9$. If we now construct $\lambda^{-1}(\tilde{\ba})$, e.g. by proceeding as described in our previous work \cite{DGO}, for example, we obtain back the valid pair represented by
$\begin{tikzpicture}[scale=.25]
\draw (0,0) node [inner sep=.25mm,minimum size=1mm] (x2){$2$};
\draw (1,0) node [inner sep=.25mm,minimum size=1mm] (x6){$6$};
\draw (2,0) node [inner sep=.25mm,minimum size=1mm] (x7){$7$};
\draw (3,0) node [inner sep=.2mm,minimum size=.5mm] (x1){$1$};
\draw (4,0) node [inner sep=.2mm,minimum size=.5mm] (x4){$4$};
\draw (5,0) node [inner sep=.25mm,minimum size=1mm] (x3){$3$};
\draw (6,0) node [inner sep=.25mm,minimum size=1mm] (x9){$9$};
\draw (7,0) node [inner sep=.25mm,minimum size=1mm] (x8){$8$};
\draw (8,0) node [inner sep=.25mm,minimum size=1mm] (x5){$5$};
\draw  [thick,looseness=0.75]  (x2) to [out=90,in=90] (x3);
\draw  [thick,looseness=0.75]  (x7) to [out=90,in=90] (x8);
\draw  [thick,looseness=0.75]  (x3) to [out=90,in=90] (x5);
\end{tikzpicture}$.
By Lemma~\ref{lem.inv}, this valid pair is $R=(\tbw,\tI)$ and the one we want to obtain is
$\widetilde{R}=(\bw,\fI)$.

\begin{remark}\label{ref}
Note that, by definition, if a point $P=(x_1, x_2,\dotsc, x_n)\in\cR$, a region of the Shi arrangement with valid pair $P=(\bw,\fI)$ (meaning that
$x_{w_1}>x_{w_2}>\dotsb> x_{w_n}$, etc.), then
$Q=( -x_n,  -x_{n-1},\dotsc, -x_1)$ belongs to the region $\widetilde{\cR}$ of valid pair $\widetilde{P}=(\tbw,\tI)$.
Hence, the map $(x_1, x_2,\dotsc, x_n)\mapsto ( -x_n,  -x_{n-1},\dotsc, -x_1)$ preserves the Shi arrangement and switches the two labelings.
\end{remark}

Since centers appeared in connection with the proceeding for inverting the Pak-Stanley labeling that we have just mentioned, a brief reference to the meaning of this concept in that context might be in order. We highlight the following result.

\begin{theorem}[{\cite[Lemmas 4.3 and 4.7, adapt.]{DGO}}]
Let, for a parking function $\ba=\lambda(\bw,\fI)\in\PF_n$ with a valid pair $(w_1\dotsb w_n,\fI)$,
$\bu=u_1u_2\dotsb u_n=w_nw_{n-1}\dotsb w_1$, 
$X=Z(\bas)$ and $m=|X|$.
If
\begin{align*}
&\bx:[m]\to X\text{ is increasing}\,,\ \bap=\ba\circ\bx\,,\\
&\bup=u_1\dotsb u_m
\ \text{ and }\ \Ip=\big\{[i,j]\in\fI\,\mid\,j\leq m\big\}\,,
\end{align*}
then:
\begin{itemize}
\item As sets, $X=\bup$.
\item $\bap=\lbdp\big(\bx^{-1}\circ\bup,\Ip\big)$, where
$\lbdp:\cR(\cS_m)\to\PF_m$ is the Pak-Stanley bijection of the Shi arrangement of size $m$.
\end{itemize}
\end{theorem}
Note that we know how to build $\bup$ directly out of $X$ and $\bap$, and how to recursively complete $(\bup,\Ip)$
onto $(\bw,\fI)$  (Cf. \cite{DGO}). In our example,
where $\bu=8\,4\,3\,9\,6\,7\,1\,2\,5\in\mathfrak{S}_9$
and $\fI=\big\{[1,6],[3,8],[6,9]\big\}$,
$\ba=\lambda(\bw,\fI)=2\,3\,{0}\,{0}\,7\,{2}\,{3}\,{0}\,{2}$ and $Z(\bas)=\{3,4,6,7,8,9\}$;
$\bup=8\,4\,3\,9\,6\,7$ and  $\bap={0}\,{0}\,{2}\,{3}\,{0}\,{2}=\lbdp(5\,2\,1\,6\,3\,4,\big\{[1,6]\big\})$.

\bigskip

\noindent\emph{Acknowledgements.}\enspace
We would like to thank Christian Krattenthaler for offering us the second proof of Theorem~\ref{bij2}, presented  here in Remark~\ref{krat}.
We would also like to thank the referee for various valuable suggestions, namely Remark~\ref{ref}.
This work was partially supported by CMUP (UID/MAT/00144/2013)
  and CIDMA (UID/MAT/04106/2013), which are funded by FCT (Portugal)
  with national (ME) and European structural funds through the
  programs FEDER, under the partnership agreement PT2020.

\end{document}